\documentclass{amsart}
\usepackage{amsmath,amssymb,hyperref}
\usepackage{amsrefs}
\usepackage{mathrsfs}
\newtheorem{thm}{Theorem}

\newtheorem{cor}[thm]{Corollary}
\newtheorem{lemma}[thm]{Lemma}
\newtheorem{prop}[thm]{Proposition}

\theoremstyle{definition}
\newtheorem{definition}[thm]{Definition}

\theoremstyle{remark}
\newtheorem{remark}[thm]{Remark}
\newtheorem{example}[thm]{Example}
\def\mathcs{C^{*}}
\newcommand{\cs}{\ensuremath{\mathcs}}

\DeclareMathSymbol{\rtimes}{\mathbin}{AMSb}{"6F}
\newcommand{\ib}{im\-prim\-i\-tiv\-ity bi\-mod\-u\-le}

\newcommand{\sme}{\,\mathord{\mathop{\text{--}}\nolimits_{\relax}}\,}

\def\C{\mathbf{C}}

\DeclareMathOperator{\Ind}{Ind}

\DeclareMathOperator*{\supp}{supp}

\DeclareMathOperator{\Aut}{Aut}
\def\set#1{\{\,#1\,\}}
\newcommand\sset[1]{\{#1\}}
\let\tensor=\otimes
\def\restr#1{|_{{#1}}}
\makeatletter
\def\labelenumi{\textnormal{(\@alph\c@enumi)}}
\def\theenumi{\@alph \c@enumi}
\def\labelenumii{\textnormal{(\@roman\c@enumii)}}
\def\theenumii{\@roman \c@enumii}
\newcount\charno
\def\alphapart#1{\charno=96
\advance\charno by#1\char\charno}

\makeatother

\def\<{\langle}
\def\>{\rangle}
\let\ipscriptstyle=\scriptscriptstyle
\def\lipsqueeze{{\mskip -3.0mu}}
\def\ripsqueeze{{\mskip -3.0mu}}
\def\ipcomma{\nobreak\mathrel{,}\nobreak}
\newbox\ipstrutbox
\setbox\ipstrutbox=\hbox{\vrule height8.5pt% depth 3.5pt
width 0pt}
\def\ipstrut{\copy\ipstrutbox}
\def\lip#1<#2,#3>{\mathopen{\relax_{\ipstrut\ipscriptstyle{
#1}}\lipsqueeze
\langle} #2\ipcomma #3 \rangle}
\def\blip#1<#2,#3>{\mathopen{\relax_{\ipstrut
\ipscriptstyle{ #1}}\lipsqueeze\bigl\langle} #2\ipcomma #3 \bigr\rangle}
\def\rip#1<#2,#3>{\langle #2\ipcomma #3
\rangle_{\ripsqueeze\ipstrut\ipscriptstyle{#1}}}
\def\brip#1<#2,#3>{\bigl\langle #2\ipcomma #3
\bigr\rangle_{\ripsqueeze\ipstrut\ipscriptstyle{#1}}}
\def\angsqueeze{\mskip -6mu}
\def\smangsqueeze{\mskip -3.7mu}
\def\trip#1<#2,#3>{\langle\smangsqueeze\langle #2\ipcomma #3
\rangle\smangsqueeze\rangle_{\ripsqueeze\ipstrut\ipscriptstyle{#1}}}
\def\btrip#1<#2,#3>{\bigl\langle\angsqueeze\bigl\langle #2\ipcomma
#3
\bigr\rangle
\angsqueeze\bigr\rangle_{\ripsqueeze\ipstrut\ipscriptstyle{#1}}}
\def\tlip#1<#2,#3>{\mathopen{\relax_{\ipstrut\ipscriptstyle{
#1}}\lipsqueeze \langle\smangsqueeze\langle} #2\ipcomma #3
\rangle\smangsqueeze\rangle}
\def\btlip#1<#2,#3>{\mathopen{\relax_{\ipstrut\ipscriptstyle{
#1}}\lipsqueeze
\bigl\langle\angsqueeze\bigl\langle} #2\ipcomma #3
\bigr\rangle\angsqueeze\bigr\rangle}

\def\ip(#1|#2){(#1\mid #2)}
\def\bip(#1|#2){\bigl(#1 \mid #2\bigr)}
\def\Bip(#1|#2){\Bigl( #1 \bigm| #2 \Bigr)}
\expandafter\ifx\csname BibSpec\endcsname\relax\else
\BibSpec{collection.article}{%
    +{}  {\PrintAuthors}                {author}
    +{,} { \textit}                     {title}
    +{.} { }                            {part}
    +{:} { \textit}                     {subtitle}
    +{,} { \PrintContributions}         {contribution}
    +{,} { \PrintConference}            {conference}
    +{}  {\PrintBook}                   {book}
    +{,} { in }                            {booktitle}
    +{,} { }                            {series}
    +{,} { \voltext}                    {volume}
    +{,} { }                            {publisher}
    +{,} { }                            {organization}
    +{,} { }                            {address}
    +{,} { \PrintDateB}                 {date}
    +{,} { pp.~}                        {pages}
    +{,} { }                            {status}
    +{,} { \PrintDOI}                   {doi}
    +{,} { available at \eprint}        {eprint}
    +{}  { \parenthesize}               {language}
    +{}  { \PrintTranslation}           {translation}
    +{;} { \PrintReprint}               {reprint}
    +{.} { }                            {note}
    +{.} {}                             {transition}
}
\BibSpec{article}{%
    +{}  {\PrintAuthors}                {author}
    +{,} { \textit}                     {title}
    +{.} { }                            {part}
    +{:} { \textit}                     {subtitle}
    +{,} { \PrintContributions}         {contribution}
    +{.} { \PrintPartials}              {partial}
    +{,} { }                            {journal}
    +{}  { \textbf}                     {volume}
    +{}  { \PrintDatePV}                {date}
    +{,} { \eprintpages}                {pages}
    +{,} { }                            {status}
    +{,} { \PrintDOI}                   {doi}
    +{,} { available at \eprint}        {eprint}
    +{}  { \parenthesize}               {language}
    +{}  { \PrintTranslation}           {translation}
    +{;} { \PrintReprint}               {reprint}
    +{.} { }                            {note}
    +{.} {}                             {transition}
}
\BibSpec{book}{%
    +{}  {\PrintPrimary}                {transition}
    +{,} { \textit}                     {title}
    +{.} { }                            {part}
    +{:} { \textit}                     {subtitle}
    +{,} { \PrintEdition}               {edition}
    +{}  { \PrintEditorsB}              {editor}
    +{,} { \PrintTranslatorsC}          {translator}
    +{,} { \PrintContributions}         {contribution}
    +{,} { }                            {series}
    +{,} { \voltext}                    {volume}
    +{,} { }                            {publisher}
    +{,} { }                            {organization}
    +{,} { }                            {address}
    +{,} { pp.~}                        {pages}
    +{,} { \PrintDateB}                 {date}
    +{,} { }                            {status}
    +{}  { \parenthesize}               {language}
    +{}  { \PrintTranslation}           {translation}
    +{;} { \PrintReprint}               {reprint}
    +{.} { }                            {note}
    +{.} {}                             {transition}
}
\fi
\newcommand\go{G^{(0)}} %% Unit spaces
\newcommand\ho{H^{(0)}}
\newcommand\lo{L^{(0)}}

\renewcommand\lg{\lambda} %% Haar systems
\newcommand\lh{\beta}
\renewcommand\ll{\kappa}
\def\g[#1,#2]{{}_{G}[#1,#2]} %% Groupoid maps on equivalences
\def\h[#1,#2]{[#1,#2]_{H}}
\newcommand\op{\text{\normalfont op}}
\newcommand\zop{Z^{\op}}

\newcommand\pg{p_{G}} % Use for bundle map
\newcommand\ph{p_{H}}
\newcommand\Pg{\mathfrak{p}_{G}} % Use for full projection
\newcommand\Ph{\mathfrak{p}_{H}}

\newcommand\Lip{\lip \scriptstyle\star}
\newcommand\Rip{\rip \scriptstyle\star}

\newcommand\X{\mathsf{X}}
\newcommand\Y{\mathsf{Y}}

\renewcommand\H{\mathcal{H}}
\DeclareMathOperator{\Lin}{Lin}
\newcommand\HH{\mathscr{H}}

\def\ipp(#1|#2){\ip({#1}|{#2})_{\pi}}

\newcommand\yind{\Y\text{-}\Ind}

\newcommand\lsp{\operatorname{span}}
\newcommand\clsp{\overline{\lsp}}
\newcommand\bundlefont\mathscr
\newcommand\B{\bundlefont{B}}
\newcommand\E{\bundlefont{E}}
\newcommand\CC{\bundlefont{C}}
\newcommand\A{\bundlefont{A}}
\newcommand\gho{G_{\ho}}
\newcommand\hg{H^{G}}
\newcommand\ssb{\sigma^{*}\B}
\newcommand\bh{\B\restr{H}}
\def\sa_#1(#2,#3){\Gamma_{#1}(#2;#3)}
\renewcommand\L{\mathcal{L}}
\newcommand\sacgb{\sa_{c}(G,\B)}
\newcommand\ripd{\rip D}
\let\phi\varphi
\newcommand\eop{\E^{\,\op}}
\newcommand\qop{q^{\op}}
\newcommand{\linke}{L(\E)}
\newcommand\linkq{L(q)}
\newcommand\ea{\mathfrak{a}}
\newcommand\eb{\mathfrak{b}}
\DeclareMathOperator{\lt}{lt}
\newcommand\half{\frac12}

\allowdisplaybreaks

\usepackage{color}
\newcount\hours
\newcount\minutes       %   For computing the time of day on
\def\timeofday{%    Must be computed when called if preloaded
\hours=\time
\minutes=\hours
\divide\hours by60
\multiply\hours by60
\advance\minutes by-\hours
\divide\hours by60
\ifnum\hours>9\else0\fi\the\hours:\ifnum\minutes>9\else
0\fi\the\minutes}
\def\predate{\date{\color{red}\bfseries \the\day\ \ifcase\month\or
  January\or February\or March\or April\or May\or June\or July\or
        August\or September\or October\or November\or
           December\fi\ \the\year\ --- \timeofday -- v001}}

\begin{document}

\author{Aidan Sims}
\address{School of Mathematics and Applied Statistics\\
University of Wollongong \\
NSW 2522\\
Australia}

\email{asims@uow.edu.au}

\author{Dana P. Williams}
\address{Department of Mathematics \\ Dartmouth College \\ Hanover, NH
03755-3551}

\email{dana.williams@Dartmouth.edu}



\date{1 December 2011}

\title[Fell Notes]{\boldmath An Equivalence Theorem for Reduced Fell
  Bundle \cs-Algebras}

\begin{abstract}
  We show that if $\E$ is an equivalence of upper semicontinuous Fell
  bundles $\B$ and $\CC$ over groupoids, then there is a linking
  bundle $\linke$ over the linking groupoid $\L$ such that the full
  cross-sectional algebra of $\linke$ contains those of $\B$ and $\CC$
  as complementary full corners, and likewise for reduced
  cross-sectional algebras. We show how our results generalise to
  groupoid crossed-products the fact, proved by Quigg and Spielberg,
  that Raeburn's symmetric imprimitivity theorem passes through the
  quotient map to reduced crossed products.
\end{abstract}

\maketitle \tableofcontents

\section{Introduction}\label{sec:intro}

The purpose of this paper is to prove a reduced equivalence theorem for cross-sectional algebras of
Fell bundles over groupoids, and to prove that the \ib{} which implements the equivalence between
the reduced $C^*$-algebras is a quotient of the Muhly-Williams equivalence bimodule between the
full $C^*$-algebras \cite{muhwil:dm08}.

An increasingly influential interpretation of Hilbert bimodules (or $C^*$-cor\-respon\-dences) is
to regard them as generalized endomorphisms of $C^*$-algebras. Imprimitivity bimodules represent
isomorphisms, and a Fell bundle over a groupoid $G$ is then the counterpart of an action of $G$ on
a $C_0(G^0)$-algebra $A$. The cross-sectional algebras of the bundle are analogues of groupoid
crossed products. For example, if $G$ is a group and each imprimitivity module is of the form
$_{\alpha}A$ for an automorphism $\alpha$ of $A$ (see, for example \cite{pim:fic87}), then the
cross-sectional algebras of such bundles are precisely those arising from group crossed products;
Fell and Doran called these semidirect products in their magnum opus \cite{fd:representations2}*{\S
VIII.4.2}. In particular, if $A = C_0(\go)$ and each fibre of the Fell-bundle is 1-dimensional,
then the cross-sectional algebras are the usual groupoid $C^*$-algebras.

The classical result which motivates this paper is that if groups $G$ and $H$ act freely, properly
and transitively on the same locally compact Hausdorff space $P$ and the actions commute, then the
groups are the same. To see why, fix $x \in P$. Then for each $g \in G$, there is a unique $h \in
H$ such that $g\cdot x = x\cdot h$, and since the actions commute, $g \mapsto h$ is an isomorphism
of $G$ with $H$. Hence $C^*(G) \cong C^*(H)$ and $C^*_r(G) \cong C^*_r(H)$. A particularly powerful
viewpoint on this is the following.  If $P^{\mathrm{op}}$ is a copy of the space $P$, but with the
actions reversed so that $G$ acts on the right and $H$ on the left, then $L = G \sqcup P \sqcup
P^{\mathrm{op}} \sqcup H$ is a groupoid, called the linking groupoid, with two units. The isotropy
at one unit is $G$ and the isotropy at the other is $H$, and conjugation in $L$ by any element of
$P$ determines an isomorphism from $G$ to $H$. At the level of $C^*$-algebras, we obtain the
following very nice picture: the actions of $G$ and $H$ on $P$ induce convolution-like products
$C_c(G) \times C_c(P) \to C_c(P)$ and $C_c(P) \times C_c(H) \to C_c(P)$, and $C_c(L)$ decomposes as
a block $2 \times 2$ matrix algebra
\[
C_c(L) \cong
\begin{pmatrix}
  C_c(G)  &  C_c(P) \\
  C_c(P^{\mathrm{op}}) & C_c(H)
\end{pmatrix}.
\]
Moreover, the universal norm on $C_c(L)$ restricts to the universal norm on each of $C_c(G)$ and
$C_c(H)$, and likewise for reduced norms. The characteristic function $1_P$ of $P$ is a partial
isometry in the multiplier algebra of each of $C^*(L)$ and $C^*_r(L)$ and conjugation by $1_P$
implements the isomorphisms $C^*(G) \cong C^*(H)$ and $C^*_r(G) \cong C^*_r(H)$.

When $G$ and $H$ do not act transitively, the actions of
$G$ and $H$ on $P$ induce actions of $G$ on $P/H$ and of $H$ on
$G\backslash P$. The picture at the level of groups is now somewhat
more complicated, but the $C^*$-algebraic picture carries over nicely:
replacing $C_c(G)$ with $C_c(G, C_c(P/H))$ and $C_c(H)$ with $C_c(H,
C_c(G \backslash P))$ in the matrix above, we obtain a $^*$-algebra
$L(P)$. The $C^*$-identity allows us to extend the norm on $C_0(P/H)
\rtimes G$ to a norm on $L(P)$. Moreover, this norm is
consistent with the norm on $C_0(G \backslash P) \rtimes H$, and the
completion of $L(P)$ in this norm contains $C_0(P/H) \rtimes G$ and
$C_0(G \backslash P) \rtimes H$ as complementary full
corners. Further, this whole apparatus descends under
quotient maps to reduced crossed products.

To prove an analogue of this equivalence theorem in the context of
Fell bundles, one uses the notion of an \emph{equivalence} of Fell
bundles specified in \cite{muhwil:dm08}. The concept is closely
modeled on the situation of groups; but the natural objects on which
Fell bundles act are Banach bundles in which the fibres are
equivalence bimodules. That is, given Fell bundles $\B$ and $\CC$ over
groupoids $G$ and $H$, an equivalence between the two is, roughly
speaking, an upper semicontinuous Banach bundle $\E$ over a space $Z$
such that $Z$ admits actions of $G$ and $H$ making it into an
equivalence of groupoids in the sense of Renault, each fibre $E_z$ of
$\E$ is an {\ib} from the fibre $B_{r(z)}$ of $\B$ over $r(z)$ to
$C_{s(z)}$, and there are fibred multiplication operations $\B * \E
\to \E$ and $\E * \CC \to \E$ which are compatible with the bundle
maps, and which implement isomorphisms $B_x \otimes_{B_u} E_z \cong
E_{x\cdot z}$. Muhly and Williams show in
\cite{muhwil:dm08}*{Theorem~6.4} that given such an equivalence, the
full cross-sectional algebras $C^*(G, \B)$ and $C^*(H, \CC)$ are
Morita equivalent (Kumjian proves the corresponding statement for
reduced $C^*$-algebras in the $r$-discrete situation in
\cite{kum:pams98}).

In this paper, we show that Muhly and Williams's Morita equivalence passes to reduced algebras. We
do so by constructing a linking bundle $\linke = \B \sqcup \E \sqcup \eop \sqcup \CC$ and showing
that $\sa_{c}(L, \linke)$ has a matrix decomposition as above. We then prove that the completion of
$\linke$ in the universal norm is a linking algebra for a Morita equivalence between $C^*(G,\B)$
and $C^*(H,\CC)$, and likewise for reduced $C^*$-algebras.

We conclude by showing how to recover a generalisation of Quigg and
Spielberg's theorem \cite{quispi:hjm92} which says that the symmetric
imprimitivity bimodule arising in Raeburn's symmetric imprimitivity
theorem \cite{rae:ma88} passes under the quotient map to an
imprimitivity bimodule for reduced crossed products.

Our reduced equivalence theorem itself is not new: late in the development of this paper, we
learned that Moutou and Tu also prove that equivalent Fell bundles have Morita equivalent reduced
cross-sectional algebras \cite{moutu:xx11}. It appears that Moutou and Tu deal only with Fell
bundles whose underlying Banach bundles are required to be continuous rather than just upper
semicontinuous.  (Upper semicontinuous bundles turn out to be the more natural object in the
context of $C^*$-algebras --- see \cite{muhwil:dm08} and especially
\cite{wil:crossed}*{Appendix~C}). Moreover Moutou and Tu restrict attention to \emph{principle}
$G$-spaces for their groupoid equivalences. But these are minor points and the arguments of
\cite{moutu:xx11} would surely go through unchanged to our setting. The main new contribution in
this article that we develop the linking bundle technology to show explicitly that the full
cross-sectional algebras of the linking bundle is a linking algebra for the full cross sectional
algebras of $\B$ and $\CC$, and that the quotient map from the full to the reduced cross-sectional
algebra of the linking bundle implements the quotients $C^*(G,\B) \to C^*_r(G, \B)$ and $C^*(H,
\CC) \to C^*_r(H, \CC)$. In particular, if $I^{\CC}_r$ is the ideal of $C^*(H, \CC)$ consisting of
elements whose reduced norm is zero, then the equivalence bimodule $X_r$ which we obtain between
reduced cross-sectional algebras is the quotient of the equivalence bimodule $X$ between full
algebras by $X \cdot I^{\CC}_r$. Consequently, induction over $X$ carries $I^\CC_r$ to the
corresponding ideal $I^{\B}_r$ of $C^*(G, \B)$.

\section{Background}\label{sec:background}

Recall that for second-countable locally compact Hausdorff groupoids
$G$ and $H$, a $G\sme H$ equivalence is a locally compact Hausdorff
space $Z$ which is simultaneously a free and proper left $G$-space and
a free and proper right $H$-space (with continuous \emph{open} fibre
maps) such that the actions of $G$ and $H$ on $Z$ commute, the map
$r_Z$ induces a homeomorphism of $Z/H$ with $\go$ and the map $s_Z$
induces a homeomorphism of $G\backslash Z$ with $\ho$. Then $G$ acts
on $Z *_r Z$ by $g \cdot (y,z) = (g\cdot y, g\cdot z)$, and the
formula $h \cdot \h[g,h] = g$ defines a homeomorphism $\h[\cdot,\cdot]
: G\backslash (Z *_r Z) \to H$; and $\g[\cdot, \cdot] : Z*_s Z \to G$
is defined similarly (see \cite{mrw:jot87}*{Definition~2.1} for
details).

Recall that an upper semicontinuous Banach bundle over a locally compact Hausdorff space $Z$ is a
topological space $\B$ together with a continuous open surjection $q : \B \to Z$ such that each
$\B_z := q^{-1}(z)$ is a Banach space and: $b \mapsto \|b\|$ is upper semicontinuous; addition is
continuous from $\B *_q \B \to \B$; scalar multiplication is continuous on $\B$; and $\|b_i\| \to
0$ and $q(b_i) \to z$ implies $b_i \to 0_z \in q^{-1}(z)$.  The concept of an upper semicontinuous
Banach bundle goes back to \cite{dg:banach}, where they were called $(H)$-bundles, and the work of
Hofmann \citelist{\cite{hofkei:lnm79} \cite{hof:lnm77}\cite{hof:74}\cite{dauhof:mams68}}. Fell
calls such bundles loose in \cite{fd:representations1}*{Remark~C.1}.  Further details and comments
concerning upper semicontinuous Banach bundles are given in \cite{muhwil:dm08}*{Appendix~A} and in
the \cs-case in \cite{wil:crossed}*{Appendix~C}. As in \cite{muhwil:dm08}, a \emph{Fell bundle}
over a locally compact Hausdorff groupoid $G$ is an upper semicontinuous Banach bundle $q : \B \to
G$ endowed with a continuous bilinear associative map $(a,b) \mapsto ab$ from $\B^{(2)} :=
\set{(a,b) \in \B \times \B : s(q(a)) = r(q(b))}$ to $\B$ such that
\begin{enumerate}
\item $q(ab) = q(a)q(b)$ for all $(a,b) \in \B^{(2)}$;
\item $q(a^*) = q(a)^{-1}$ for all $a \in \B$;
\item $(ab)^* = b^*a^*$ for all $(a,b) \in \B^{(2)}$;
\item for each $u \in \go$, the fibre $A_u := q^{-1}(u)$ is a
  $\cs$-algebra under these operations; and
\item for each $g \in G \setminus \go$, the fibre $B_g := q^{-1}(g)$
  is an $A_{r(g)}\sme A_{s(g)}$-{\ib} with actions determined by
  multiplication in $\B$ and inner products $\lip<a,b> = a^*b$ and
  $\rip<a,b> = ab^*$.
\end{enumerate}
As a notational convenience, we define $r,s : \B \to \go$ by $r(a) :=
r_G(q(a))$ and $s(a) := s_G(q(a))$. See \cite{muhwil:dm08} for more
details regarding Fell bundles over groupoids.

\begin{remark} In the context of bundles over groups, the
fibres in a Fell bundle are not always assumed to be imprimitivity bimodules (they are not assumed
to be full --- see \cite{kum:pams98}*{2.4}). Bundles in which all the fibres are indeed
imprimitivity bimodules are then called \emph{saturated}. We take this condition as part of our
definition.  It should also be observed that the underlying Banach bundle of a Fell bundle over a
\emph{group} is always continuous \cite{bmz}*{Lemma~3.30}.
\end{remark}

\begin{remark}\label{rmk:why notation}
In our notation the fibre of $\B$ over a unit $u$ can be denoted either $A_u$ or $B_u$. The dual
notation allows us to emphasise its dual roles. We write $A_u$ to emphasise its role as a
$C^*$-algebra, and $B_u$ to emphasise its role as an {\ib}.  The \cs-algebra $A:=\sa_{0}(\go,\B)$
is called the \emph{\cs-algebra of the Fell bundle $\B$ over $\go$}.
\end{remark}

We recall from \cite{muhwil:dm08} the definition of an equivalence of
Fell bundles. First, fix a second-countable locally compact Hausdorff
groupoid $G$, a left $G$-space $Z$, a Fell bundle $q_G : \B \to G$,
and a Hausdorff space $\E$ together with a continuous open surjection
$q : \E \to Z$. Again, as a notational convenience, we shall write $r$
for the composition $r_Z \circ q : \E \to \go$. we say that $\B$
\emph{acts on the left of $\E$} if there is a pairing $(b, e) \mapsto
b\cdot e$ from $\B * \E = \set{(b,e) \in \B \times \E : s(b) = r(e)}$
to $\E$ such that
\begin{enumerate}
\item $q(b\cdot e) = q_G(b)q(e)$ for $(b,e) \in \B * \E$;
\item $a\cdot(b\cdot e) = (ab)\cdot e$ whenever $(a,b) \in \B^{(2)}$
  and $(b,e) \in \B * \E$;
\item $\|b \cdot e\| \le \|b\| \|e\|$ for $(b,e) \in \B *
  \E$.\footnote{The equality appearing in the corresponding item in
    \cite{muhwil:dm08} is a typographical error.}
\end{enumerate}
If $\E$ is a right $H$-space, and $q_H : \CC \to H$ is a Fell bundle,
then a right action of $\CC$ on $\E$ is defined similarly.

Now fix second-countable locally compact Hausdorff groupoids $G$ and
$H$ and a $G\sme H$ equivalence $Z$. Suppose that $q_G : \B \to G$ and
$q_H : \CC \to H$ are Fell bundles. Fix a Banach bundle $q : \E \to
Z$. We write $\E *_s \E$ for $\set{(e,g) \in \E \times \E : s(e) =
  s(g)}$ and we define $\E *_r \E$ similarly. We call $\E$ a $\B\sme
\CC$ equivalence if:
\begin{enumerate}
\item\label{it:FBE actions} there are a left action of $\B$ on $\E$
  and a right action of $\CC$ on $\E$ which commute;
\item\label{it:FBE ips} there are sesquilinear maps
  $\lip\B<\cdot,\cdot> : \E *_s \E \to \B$ and $\rip\CC<\cdot,\cdot> :
  \E *_r \E \to \CC$ such that the relations
  \begin{enumerate}
  \item\label{it:FBE ips1} $q_G(\lip\B<e,f>) = \g[q(e),q(f)]$ and
    $q_H(\rip\CC<e,f>) = \h[q(e),q(f)]$,
  \item\label{it:FBE ips2} $\lip\B<e,f>^* = \lip\B<f,e>$ and
    $\rip\CC<e,f>^* = \rip\CC<f,e>$,
  \item\label{it:FBE ips3} $b \lip\B<e,f> = \lip\B<b\cdot e, f>$ and
    $\rip\CC<e,f> c = \rip\CC<e,f\cdot c>$ and
  \item\label{it:FBE ips4} $\lip\B<e,f>\cdot g = e \cdot \rip\CC<f,g>$
  \end{enumerate}
  are satisfied whenever they make sense; and
\item\label{it:FBE imprim} under the actions described in~(\ref{it:FBE
    actions}) and the inner-products defined in~(\ref{it:FBE ips}),
  each $\E_z := q^{-1}(z)$ is an $A_{r(z)}\sme D_{s(z)}$-{\ib}.
\end{enumerate}

As in \cite{simwil:jot11}, if $G,H$ are second-countable locally
compact Hausdorff groupoids with Haar systems $\lg, \lh$ and $Z$ is a
$G \sme H$ equivalence, we write $\zop$ for the ``opposite
equivalence'' $\zop = \set{\bar{z} : z \in Z}$ with $r(\bar{z}) =
s(z)$, $s(\bar{z}) = r(z)$, $h \cdot \bar z := \overline{z \cdot
  h^{-1}}$ and $\bar z \cdot g := \overline{g^{-1} \cdot z}$. Then $L
:= G \sqcup Z \sqcup \zop \sqcup H$ with $\lo := \go \sqcup \ho
\subseteq L$ is a groupoid containing $G$ and $H$ as subgroupoids: we
extend the inverse map to $Z$ and $\zop$ by setting $z^{-1}
:= \bar{z}$; and multiplication between $Z$ and $G,H$ is implemented
by the left and right actions, while multiplication between $Z$ and
$\zop$ is implemented by $\g[\cdot,\cdot]$ and $\h[\cdot,\cdot]$. See
\cite{simwil:jot11}*{Lemma~5} for details. There is a Haar system on
$L$ determined by
\[
\ll^w(F) := \begin{cases}
  \textstyle\int_{G} F(g)\,d\lg^{w}(g) + \int_{H} F(z\cdot
  h)\,d\lh^{s(z)}(h) &\text{ if $w \in \go$}\vadjust{\vskip0.5em}\\
  \textstyle\int_{G} F(\bar{y}\cdot g)\,d\lg^{s(\bar{y})}(g) +
  \int_{H} F(h)\,d\lh^{w}(h) &\text{ if $w \in \ho$}
\end{cases}
\]
for $F \in C_c(L)$ and $w \in \lo$ (see \cite{simwil:jot11}*{Lemma~6}). For $u \in \go$ and $v \in
\ho$, we write $\sigma^u_Z$ and $\sigma^v_{\zop}$ for the restrictions of $\ll^u$ to $Z$ and of
$\ll^v$ to $\zop$. The main results of \cite{simwil:jot11} say that $\cs(L,\ll)$ contains
$\cs(G,\lg)$ and $\cs(H,\lh)$ as the complementary full corners determined by the multiplier
projections $1_{\go}$ and $1_{\ho}$, and that this Morita equivalence passes under the quotient map
$\cs(L,\ll) \to \cs_r(L,\ll)$ to reduced groupoid \cs-algebras. Our goal in this article is to
establish the corresponding statement for Fell bundles. As a first step, we show in the next
section how to construct from an equivalence of Fell bundles a linking bundle over the linking
groupoid.

\section{Linking bundles}\label{sec:linking bundle}

Let $G$ and $H$ be locally compact Hausdorff groupoids, let $Z$ be a
$G\sme H$ equivalence, and let $L$ be the linking groupoid as
above. Suppose that $\pg : \B \to G$ and $\ph : \CC \to H$ are
upper-semicontinuous Fell bundles, and that $q : \E \to Z$ is a bundle
equivalence. We denote by $A$ the \cs-algebra $\sa_{0}(\go,
p_G^{-1}(\go))$ of the bundle $\B$, and by $D$ the \cs-algebra
$\sa_{0}(\ho, p_H^{-1}(\ho))$ of $\CC$; so the fibre over $u \in \go$
is $A_u$, the fibre over $v \in \ho$ is $D_v$, and each $\E_z$ is an
$A_{r(z)}\sme D_{s(z)}$-{\ib}.

Let $\eop=\set{\bar e:e\in\E}$ be a copy of the topological space $\E$
endowed with the conjugate Banach space structure $\alpha \bar e +
\bar f = \overline{(\overline{\alpha} e + f)}$ on each fibre. Then
$\qop:\eop\to\zop$ is an upper-semicontinuous Banach bundle with
$\qop(\bar e)=\overline{q(e)}$.  We have $s(\bar e) = s(\qop(\bar e))
= r(e)$ and likewise $r(\bar e) = s(e)$, so we obtain a right
$\B$-action and a left $\CC$-action on $\eop$ by
\begin{equation}\label{eq:eop actions}
  \bar e\cdot b=\overline{b^{*}\cdot e} \quad\text{ and }\quad c\cdot \bar e =\overline{e\cdot c^{*}}.
\end{equation}
The inner products on $\eop*_{r}\eop$ and $\eop *_{s} \eop$ are given
by $\rip\B<\bar e,\bar f>=\lip\B<e,f>$ and $\lip\CC<\bar e,\bar
f>=\rip\CC<e,f>$.  Routine calculations show that each $E^{\op}(\bar
z)$ is the dual {\ib} $E(z)^{\sim}$ of $E(z)$. Since $s(z)=r(\bar z)$
and $r(z)=s(\bar z)$, axioms (a), (b)~and~(c) of
\cite{muhwil:dm08}*{Definition~6.1} hold, so $\eop$ is a
$\CC\sme\B$-equivalence.

Let $\linke=\B\sqcup\E\sqcup\eop\sqcup\CC$ and define
$\linkq:\linke\to L$ by
\[
\linkq|_{\B} = \pg, \quad \linkq|_{\CC} = \ph,\quad \linkq|_{\E} =
q\quad\text{ and }\quad \linkq|_{\eop} = \qop.
\]
Since $e \mapsto \bar e$ is a fiberwise-isometric homeomorphism from
$\E$ to $\eop$ and since $z \mapsto \bar z$ is a homeomorphism from
$Z$ to $\zop$, the bundle $\linke$ is an upper semicontinuous Banach
bundle. Let
\begin{equation*}
  \linke^{(2)}=\set{(\ea,\eb)\in \linke\times\linke:
    s(\linkq(\ea))=r(\linkq(\eb))}.
\end{equation*}
Define $m:\linke^{(2)}\to\linke$ to coincide with the given
multiplications on $\B$ and $\CC$ and with the actions of $\B$ and
$\CC$ on $\E$ and $\eop$, and to satisfy
\begin{equation*}
  m(e, \bar f)=\lip\B<e,f>\text{ for
      $(e,f)\in\E*_{s}\E$}
    \quad\text{and}\quad
    m(\bar e, f)=\rip\CC<e,f>\text{ for $(e,f)\in\E*_{r}\E$.}
\end{equation*}
We define $\ea\mapsto \ea^{*}$ on $\linke$ to extend the given
involutions on $\B$ and $\CC$ by setting $e^{*}=\bar e$ on $\E$ and
${\bar e}^{*}=e$ on $\eop$.

\begin{lemma}
  \label{lem:linking bundle}
  With notation as above, the bundle $\linke$ is a Fell bundle over
  $L$. Moreover, the \cs-algebra $\sa_0(\lo, \linkq^{-1}(\lo))$ is
  isomorphic to $A \oplus D$.
\end{lemma}
\begin{proof}
  We know already that $\linke$ is an upper-semicontinuous Banach
  bundle, that each $\linke_u$ is a \cs-algebra and each $\linke_x$ is
  a $\linke_{r(x)}\sme \linke_{s(x)}$-{\ib}. The fibre map $q$
  preserves multiplication and involution by definition of these
  operations. The operations are continuous because they are
  continuous on each component of $\linke$ and of $\linke * \linke$,
  and the components are topologically disjoint. That $(\ea\eb)^* =
  \eb^*\ea^*$ is clear on $\B * \B$ and $\CC * \CC$, follows from the
  inner-product axioms on $\E * \eop$ an $\eop * \E$, and follows
  from~\eqref{eq:eop actions} for the remaining
  pairings. Associativity for triples from $\E * \eop * \E$ and $\eop
  * \E * \eop$ follows from the {\ib} axiom $\Lip<e,f> g = e
  \Rip<f,g>$, and is clear for all other triples.

  The map $f \mapsto (f|_{\go}, f|_{\ho})$ is a surjection $\sa_0(\lo,
  \linkq^{-1}(\lo)) \to A \oplus D$, and the inverse makes sense
  because $\go$ and $\ho$ are topologically disjoint. Hence
  $\sa_0(\lo, \linkq^{-1}(\lo)) \cong A \oplus D$
\end{proof}

Resume the hypotheses of Lemma~\ref{lem:linking bundle}. It is routine
to check that $(\Pg \phi)(g) := \chi_{\go}(r(g)) \phi(g)$ determines a
bounded self-adjoint map on $\sacgb$ under the inner-product
$(\phi,\psi) \mapsto \phi^*\psi$, and hence extends to a multiplier
projection\label{page:multiplier}, also denoted $\Pg$, of
$\cs(G;\B)$. Taking adjoints, $(\phi \Pg)(\ea) = \chi_{\go}(s(\ea))
\phi(\ea)$. The corresponding projection $\Ph$ for $H$ is defined
similarly.

\begin{remark}
  As in \cite{simwil:jot11}, we think of $\varphi \in \Gamma_c(L,
  \linke)$ as a matrix
  \[
  \begin{pmatrix} \varphi_{G} & \varphi_{Z} \\ \varphi_{\zop} &
    \varphi_{H}\end{pmatrix}
  \]
  where $\varphi_{G}$ is the restriction of $\varphi$ to $G \subseteq
  L$ and similarly for the other terms. With respect to this
  decomposition, we have
  \[
  \varphi\psi =
  \begin{pmatrix} \varphi_G \psi_G + \varphi_Z \psi_{\zop} & \varphi_G
    \psi_Z + \varphi_Z \psi_H \\
    \varphi_{\zop} \psi_G + \varphi_H \psi_Z & \varphi_{\zop} \psi_Z +
    \varphi_H \psi_H \end{pmatrix},
  \]
where we have used juxtaposition for the convolution product restricted to the various
corners.\footnote{In fact, the products in the matrix can be expressed in terms of the
inner-products and module actions from \cite{muhwil:dm08}*{Theorem~6.4}.} Moreover $\varphi_G = \Pg
\varphi \Pg$, $\varphi_Z = \Pg \varphi \Ph$, $\varphi_{\zop} = \Ph \varphi \Pg$, and $\varphi_H =
\Ph \varphi \Ph$.
\end{remark}

\begin{lemma}
\label{lem:full corners}
Resume the hypotheses of Lemma~\ref{lem:linking bundle}. Then $\Pg$
and $\Ph$ are full multiplier projections of $\cs(L; \linke)$.
\end{lemma}
\begin{proof}
We just show that $\Pg$ is full; the corresponding statement for $\Ph$ follows by symmetry. Fix
$\varphi, \psi \in \Gamma_c(L;\linke)$. Using the matrix notation established above, we have
\[
\varphi \Pg \psi =
\begin{pmatrix}
  \varphi_G \psi_G & \varphi_G \psi_Z \\
  \varphi_{\zop} \psi_G & \varphi_{\zop}\psi_Z
\end{pmatrix}.
\]
That elements of the form $\varphi_G \psi_G$ span a dense subalgebra
of $\Gamma_c(G; \B)$ is clear.  That elements of the form $\varphi_G
\psi_Z$ span a dense subspace of $\Gamma_c(Z; \E)$ and likewise that
elements of the form $\varphi_{\zop} \psi_G$ span a dense subspace
of $\Gamma_c(\zop; \eop)$ follows from
\cite{muhwil:dm08}*{Proposition~6.10}. That elements of the form
$\varphi_{\zop}\psi_Z$ span a dense subspace of $\Gamma_c(H;\CC)$
follows from the argument which establishes axiom~(IB2) in
\cite{muhwil:dm08}*{Section~7}.
\end{proof}

Recall that the inductive-limit topology on $C_c(X)$ for a locally compact Hausdorff space $X$ is
the unique finest locally convex topology such that for each compact $K \subseteq X$, the inclusion
of $C_c(X)^K = \{f \in C_c(X) : \supp(f) \subseteq K\}$ into $C_c(X)$ is continuous (see for
example \cite{fd:representations1}*{II.14.3} or \cite{rw:morita}*{\S D.2}). In particular,
\cite{rw:morita}*{Lemma~D.10} says that to check that a linear map $L$ from $C_{c}(X)$ into any
locally convex space $M$ is continuous, it suffices to see that if $f_{n}\to f$ uniformly and if
all the supports of the $f_{n}$ are contained in the same compact set $K$, then $L(f_{n})\to L(f)$.

\begin{remark}
  \label{rem:ourproof}
  We are now in a situation analogous to that of
  \cite{simwil:jot11}*{Remark~8}.  By the Disintegration Theorem for
  Fell bundles, \cite{muhwil:dm08}*{Theorem~4.13}, any pre-$\cs$-norm
  $\|\cdot\|_\alpha$ on $\Gamma_c(L;\linke)$ which is continuous in
  the inductive-limit topology is dominated by the universal
  norm. Hence the argument of \cite{simwil:jot11}*{Remark~8} shows that
  $\Pg \cs_\alpha(L;\linke) \Ph$ is a $\cs_\alpha(G;\B) \sme
  \cs_\alpha(H;\CC)$-imprimitivity bimodule. So to prove that
  $\cs(G;\B)$ is Morita equivalent to $\cs(H, \CC)$ we just need to
  show that for $F \in \Pg \sa_{c}(L,\linke) \Pg$, the universal norms
  $\|F\|_{\cs(L,\linke)}$ and $\big\|F|_G\big\|_{\cs(G;\B)}$ coincide,
  and similarly for the reduced algebras (the corresponding statements
  for $H$ hold by symmetry).
\end{remark}

\section{The reduced norm}

In this section we recall the construction of the reduced
cross-sectional algebra of a Fell bundle.  We first discuss how to
induce representations from \cs-algebra of the restriction of a Fell
bundle to a closed subgroupoid up to representations of the
\cs-algebra of the whole bundle. We then apply this construction to
the closed subgroupoid $\go$ of $G$ to induce representations of the
$C^*$-algebra $A = \sa_{0}(G^{(0)}, \B|_{G{(0)}})$ up to
representations of $\cs(G;\B)$. These are, by definition, the regular
representations whose supremum determines the reduced norm.

\subsection{Induced representations}\label{sec:induc-repr}

Let $G$ be a second countable locally compact Hausdorff groupoid with Haar system
$\sset{\lambda^{u}}_{u\in\go}$.  Let $q : \B\to G$ be a separable Fell bundle as described in
\cite{kmqw:nyjm10}*{\S1.3}. Assume that $H$ is a closed subgroupoid of $G$ with Haar system
$\sset{\alpha^{u}}_{u\in\ho}$.  We write $q_{H}:\B\restr H\to H$ for the Fell bundle obtained by
restriction to $H$.  We want to induce representations of $\cs(H,\B\restr H)$ to $\cs(G;\B)$ using
the Equivalence Theorem \cite{muhwil:dm08}*{Theorem~6.4} for Fell bundles.  We will use the set-up
and notation from \cite{ionwil:pams08}*{\S2}.  In particular, we recall that $\gho=s^{-1}(\ho)$ is
a $(\hg,H)$-equivalence where $\hg$ is the \emph{imprimitivity groupoid} $(\gho*_{s}\gho)/H$.  Let
$\sigma:\hg\to G$ be the continuous map given by $\sigma\bigl([x,y]\bigr) =xy^{-1}$.  The pull-back
Fell bundle $\sigma^{*}q:\sigma^{*}\B\to \hg$ is the Fell bundle $\sigma^{*}\B = \{([x,y], b) :
[x,y] \in \hg, b \in \B, \sigma([x,y]) = q(b)\}$ with bundle map $\sigma^*([x,y], b) = [x,y]$ over
$\hg$.

Let $\E=q^{-1}(\gho)$; then $q$ restricts to a map $q:\E\to \gho$. We
wish to make this Banach bundle into a $\ssb\sme\bh$-equivalence (see
\cite{muhwil:dm08}*{Definition~6.1}).  It is clear how $\bh$ acts on
the right of $\E$, and we get a left action of $\ssb$ via
\begin{equation*}
  \bigl([x,y],b\bigr)\cdot e:=be\quad\text{for $q(e)=yh$.}
\end{equation*}
(Since $q(b)=xy^{-1}$, $q(be)=xh$ as required.) The ``inner products''
on $\E*_{r}\E$ and $\E*_{s}\E$ are given by
\begin{equation*}
  \rip\bh<e,f>=e^{*}f\quad\text{and}\quad \lip\ssb<e,f>=\bigl([q(e),
  q(f)], ef^{*}\bigr),
\end{equation*}
respectively.  It now straightforward to check that $\E$ is a
$\ssb\sme\bh$-equivalence.  By \cite{muhwil:dm08}*{Theorem~6.4},
$\sa_{c}(\gho,\E)$ is a pre-{\ib} with actions and inner products
determined by
\begin{align}
  F\cdot \phi(z)&=\int_{G}
  F\bigl([z,y]\bigr)\phi(y)\,d\lambda_{s(z)}(y),\label{eq:1}
  \\
  \phi\cdot g(z)&=\int_{H}
  \phi(zh)g(h^{-1})\,d\alpha^{s(z)}(h),\label{eq:2} \\
  \Rip<\phi,\psi>(h)
  &=\int_{G}\phi(y)^{*}\psi(yh)\,d\lambda_{r(h)}(y),\label{eq:3}
  \\
  \Lip<\phi,\psi>\bigl([x,y]\bigr)&=
  \int_{G}\phi(xh)\psi(yh)^{*}\,d\alpha^{s(x)}(h)\label{eq:4}
\end{align}
for $F\in\sa_{c}(\hg,\ssb)$, $\phi,\psi\in \sa_{c}(\gho,\E)$ and $g\in
\sa_{c}(H, \bh)$.  The completion $\X=\X_{H}^{G}$ is a
$\cs(\hg,\ssb)\sme\cs(H,\bh)$-{\ib}.
\begin{remark}
  \label{rem-oops}
  It is pleasing to note that the formalism of Fell bundles is such
  that equations \eqref{eq:1}--\eqref{eq:4} are virtually identical to
  those in the scalar case: see
  \cite{ionwil:pams08}*{Eq. (1)--(4)}.\footnote{Well, they would be if
    it weren't for the typos in equations (1)~and (4) in
    \cite{ionwil:pams08}.}  The only difference is that complex
  conjugates in the scalar case are replaced by adjoints.
\end{remark}

To construct induced representations using the machinery of
\cite{rw:morita}*{Proposition~2.66}, we need a nondegenerate
homomorphism $V : \cs(G;\B) \to \L(\X)$ which will make
$\X$ into a right Hilbert $\cs(G;\B)\sme\cs(H,\bh)$-bimodule (the data
needed to induce representations \emph{a la} Rieffel.)  Define
$V:\sacgb \to \Lin(\sa_{c}(\gho,\E))$ by
\begin{equation}\label{eq:V-def}
  V(f)(\phi):=\int_{G}f(y)\phi(y^{-1}z)\,d\lambda^{r(z)}(y).
\end{equation}
By the Tietz Extension Theorem for upper semicontinuous Banach bundles
\cite{muhwil:dm08}*{Proposition~A.5}, each $\phi \in \sa_{c}(\gho,\E)$
is the restriction of an element of $\sacgb$. So the argument of
\cite{ionwil:pams08}*{Remark~1} and the paragraph which follows
yields\begin{equation*} \Rip<V(f)\phi,\psi>=\Rip<\phi,V(f^{*})\psi>.
\end{equation*}
The map $f\mapsto \Rip<V(f)\phi,\psi>$ is continuous in the
inductive-limit topology, and the existence of approximate units in
$\sacgb$ implies that
\[
\set{V(f)\phi:\text{$f\in \sacgb$ and $\phi\in\sa_{c}(\gho,\E)$}}
\]
spans a dense subspace of $\sa_{c}(\gho,\E)$.  Then
\cite{kmqw:nyjm10}*{Proposition~1.7} implies that $V$ is bounded and
extends to a nondegenerate homomorphism as required.

Now if $L$ is a representation of $\cs(H,\bh)$, then the induced
representation $\Ind_{H}^{G}L$ of $\cs(G;\B)$ acts on the completion
of $\X\odot \H_{L}$ with respect to
\begin{equation*}
  \ip(\phi\tensor h|\psi\tensor k)=
  \bip(L(\Rip<\psi,\phi>)h|k)_{\H_{L}}.
\end{equation*}
Fix $\phi\in\sa_{c}(\gho,\E)$. Writing $f \cdot \phi$ for $V(f)\phi$,
we have
\begin{equation*}
  (\Ind_{H}^{G}L)(f)(\phi\tensor h)=f\cdot\phi\tensor h,
\end{equation*}
and, as in \cite{ionwil:pams08}*{Remark~1}, $f\cdot \phi=f*\phi$.

\subsection{Regular Representations and the reduced
  \texorpdfstring{\cs}{C*}-algebra}
\label{sec:regul-repr-cs_rg}

\emph{Regular representations} are, by definition, those induced from $A=\sa_{0}(\go,\B)$.  Thus,
in the notation of Section~\ref{sec:induc-repr}, $H=\go$, $\gho=G$, $\E=\B$, and we write $\A$ in
place of $\B\restr{\go}$ (see Remark~\ref{rmk:why notation}); in particular each $B(x)$ is a
$A(r(x))\sme A(s(x))$-{\ib}. We also have $\Rip<\phi,\psi>={\phi^{*}\psi\restr{\go}}$, with the
product being computed in $\sa_{c}(L,\linke)$.

Let $\pi$ be a representation of $A$ on $H_\pi$. Let $\tilde{\pi}$ be
the extension of $\pi$ to $M(A)$, and let $i: C_0(\go) \to M(A)$ be
the map characterised by $\big(i(f)a\big)(u) = f(u)a(u)$ for $u \in
\go$. Then $\phi := \tilde{\pi} \circ i$ is a representation of
$C_0(\go)$ on $H_\pi$ which commutes with $\pi$. Example~F.25 of
\cite{wil:crossed} shows that there is a Borel Hilbert bundle $\go *
\HH$ and a finite Radon measure $\mu$ on $\go$ such that $\pi$ is
equivalent to a direct integral $\int^\oplus_{\go} \pi_u \,d\mu$, and
such that if $L : f \to L_f$ is the diagonal inclusion of $C_0(\go)$
in $B(L^2(\go * \HH, \mu))$, then $\pi(i(f)a) = L_f \pi(a)$ for $a \in
A$. So each $\pi_u$ factors through $A_u$. We will usually write
$\pi_{u}\bigl(a(u)\bigr)$ in place of $\pi_{u}(a)$ for $a\in A$. See
\cite{muhwil:nyjm08}*{p.~46} for more details. The regular
representation $\Ind \pi=\Ind_{\go}^{G} \pi$ then acts on the
completion of $\sacgb\odot L^{2}(\go*\HH,\mu)$ with respect to
$\bip(\phi\tensor h|\psi\tensor k) =
\bip(\pi\bigl(\psi^{*}*\phi\bigr)h|k)$, and a quick calculation yields
\begin{equation}
  \label{eq:6}
  \bip(\phi\tensor h|\psi\tensor k)
  = \int_{\go}\int_{G} \bip(\pi_{u}\bigl(\psi(x)^{*}\phi(x)\bigr)h(u) |{k(u)})\,d\lambda_{u}(x)\,d\mu(u).
\end{equation}
Then $\Ind\pi$ acts by:
\begin{equation}
  \label{eq:7}
  (\Ind\pi)(f)(\phi\tensor h)=V(f)(\phi)\tensor h= f\cdot\phi\tensor h.
\end{equation}

We next define the reduced algebra of a Fell bundle. We define the
reduced norm by analogy with the one-dimensional case as the supremum
of the norms determined by induced representations of $A$. We then
show that this agrees, via $V$, with the operator norm on
$\L(\X)$. This is equivalent to Definition~2.4 and Lemma~2.7 of
\cite{moutu:xx11}, though the roles of definition and lemma are
interchanged.

\begin{definition}
  We define the \emph{reduced norm} on $\sacgb$ by
  \begin{equation*}
    \|f\|_{r}:=\sup\set{\|(\Ind\pi)(f)\|:\text{$\pi$ is a
        representation of $A$}}.
  \end{equation*}
\end{definition}

Since the kernel of $\Ind\pi$ depends only on the kernel of $\pi$ (see
\cite{rw:morita}*{Corollary~2.73}), we have $\|f\|_{r}=\|(\Ind\pi)(f)\|$ for any faithful
representation $\pi$ of $A$. We define the \emph{reduced \cs-algebra} $\cs_{r}(G;\B)$ of $\B$ to be
the quotient of $\cs(G;\B)$ by $I_{\cs_{r}(G,\B)}:=\set{a\in\cs(G,B):\|a\|_{r}=0}$.

\begin{lemma}
  \label{lem-x-red}
  Let $\X=\X_{\go}^{G}$ and $V:\cs(G;\B)\to \L(\X)$ the homomorphism
  determined by~\eqref{eq:V-def}.  Then $\ker V=I_{\cs_{r}(G;\B)}$ and
  $V$ factors through an injection of $\cs_{r}(G;\B)$ into $\L(\X)$.
  In particular, $\|V(f)\|=\|f\|_{r}$.
\end{lemma}
\begin{proof}
  Let $\pi$ be a faithful representation of $A$.  Then for any $x\in
  \X$, $h\in\H_{\pi}$ and $f\in\cs(G;\B)$, we have
  \begin{equation}
    \label{eq:12}
    \bigl\|(\Ind\pi)(f)(x\tensor h)\bigr\|^{2}=\bigl\|V(f)(x)\tensor
    h\bigr\|^{2} = \bip(\pi\bigl(\Rip<V(f)(x),V(f)(x)>\bigr)h|h).
  \end{equation}
  Thus if $V(f)=0$, then $(\Ind\pi)(f)=0$.  On the other hand, given
  $x$ and $f$, we can find a unit vector $h$ such that the right-hand
  side of \eqref{eq:12} is at least
  \begin{equation*}
    \frac12\bigl\|\pi\bigl(\Rip<V(f)(x),V(f)(x)>\bigr)\bigr\|=\frac12
    \bigl\|V(f)(x)\bigr\|^{2}.
  \end{equation*}
  Therefore $\Ind\pi(f)=0$ implies that $V(f)=0$.  We have shown that
  $\ker V=\ker(\Ind\pi)$, and hence $V$ factors through an injection
  of $\cs_{r}(G;\B)$ into $\L(X)$ as claimed.
\end{proof}

We digress briefly to check that the definition of the reduced
$C^*$-algebra which we have given is compatible with existing
definitions on some special cases.

\begin{example}[The Scalar Case: Groupoid \cs-Algebras]
  \label{ex-scalar}
  Let $\B=G\times\C$ so that $\cs(G;\B)=\cs(G)$.  So $A=C_{0}(\go)$,
  and $\pi$ defined by multiplication on $L^{2}(\go,\mu)$ is a
  faithful representation of $A$. Then $\Ind\pi$ acts on the
  completion $H$ of $C_{c}(G) \odot L^{2}(\go)$ and, if we let $\nu =
  \mu \circ \lambda$, then \eqref{eq:6} becomes
  \begin{equation*}
    \bip(\phi\tensor h|\psi\tensor k)= \int_{G}
    \bip(\phi(x)h(s(x))|{\psi(x)k(s(x))} )\,d\nu^{-1}(x).
  \end{equation*}
  Hence there is a unitary $U$ from $H$ onto $L^{2}(G,\nu^{-1})$ defined by $U(\phi\tensor h)(x)=\phi(x)h(s(x))$,
  and $U$ intertwines $\Ind\pi$ with the representation
  $(\Ind\mu)(f)\xi(x)=\int_{G}f(y)\xi(y^{-1}x)\,d\lambda^{r(x)}(y)$. Hence
  our definition of the reduced norm agrees with the usual definition
  (see \cite{simwil:jot11}*{\S3}, for example), and
  $\cs_{r}(G\times\C)=\cs_{r}(G)$.
\end{example}

\begin{example}[Groupoid Crossed Products]
  \label{ex-crossed-prod}
  Suppose that $(\A,G,\alpha)$ is a dynamical system and form the
  associated semidirect product Fell bundle $\B=r^{*}\A$ as in
  \cite{muhwil:dm08}*{Example~2.1}.  Working with the appropriate
  $\A$-valued functions, as in \cite{muhwil:dm08}*{Example~2.8}, a
  quick calculation starting from~\eqref{eq:6} gives
  \begin{multline}
    \bip(f\tensor h|g\tensor k)
    = \\ \int_{\go}\int_{G} \bip(\pi_{u}\bigl(\alpha_{x}^{-1}
    \bigl(f(x)\bigr)\bigr)h(u) |
    {\pi_{u}\bigl(\alpha_{x}^{-1}\bigl(g(x)\bigr)\bigr) k(u)} ) \,
    d\lambda_{u}(x)\,d\mu(u)\label{eq:10}.
  \end{multline}
  (Since $\lambda_{u}$ is supported on $G_u$, each
  $\alpha_x^{-1}(g(x)) \in A_u$, so the integrand makes sense.)  Let
  $G*\HH_{s}$ be the pull back of $\go*\HH$ via $s$. Given a
  representation $\pi$ of $A$, there is a unitary $U$ from the space
  of $\Ind\pi$ to $L^{2}(G*\HH_{s},\nu^{-1})$ defined by $U(f\tensor
  h)=\pi_{s(x)}\bigl(\alpha_{x}^{-1}\bigl(f(x)\bigr)\bigr)h\bigl(s(x)\bigr)$. This
  $U$ intertwines $\Ind\pi$ with the representation $L^{\pi}$ given by
  \begin{equation}\label{eq:9}
    L^{\pi}(f)\xi(x)=\int_{G}
    \pi_{s(x)}\bigl(\alpha_{x}^{-1}\bigl(f(y)\bigr)\bigr)
    \xi(y^{-1}x)\,d\lambda^{r(x)} (y).
  \end{equation}
  Applying this with a faithful representation $\pi$ of $A$, we deduce
  that $\A \rtimes_{\alpha,r} G \cong \cs_{r}(G, r^*\A)$.
\end{example}

\begin{remark}
\label{rem-problem}
In Examples \ref{ex-scalar}~and \ref{ex-crossed-prod}, the essential
step in finding a concrete realization of the space of $\Ind\pi$ is
to ``distribute the $\pi_{u}$'' in the integrand in \eqref{eq:6} to
both sides of the inner product.  But for general Fell bundles,
$\pi_{u}(\phi(x))$ makes no sense for general $x\in G_{u}$.  This
often makes analyzing regular representations of Fell bundle
\cs-algebras considerably more challenging.
\end{remark}

\begin{example}
  \label{ex-pointmass}
  Any representation $\pi_u$ of $A_u$ determines a representation
  $\pi_u \circ \epsilon_u$ of $A$ by composition with evaluation at
  $u$ (in the direct-integral picture, $\pi = \pi_u \circ \epsilon_u$
  is a direct integral with respect to the point-mass $\delta_u$). We
  abuse notation slightly and write $\Ind\pi_{u}$ for $\Ind(\pi_u
  \circ\epsilon_u)$ which acts on the completion of
  $\sacgb\odot\H_{\pi_{u}}$ under
  \begin{equation}\label{eq:8}
    \ip(\phi\tensor h|\psi\tensor
    k)=\int_{G}\bip(\pi_{u}\bigl(\psi(x)^{*}\phi(x) \bigr)h(u)|{k(u)})
    \,d\lambda_{u}(x).
  \end{equation}
  Equation~\eqref{eq:8} depends only on $\phi\restr{G_{u}}$ and
  $\psi\restr{G_{u}}$; and conversely each element of
  $\sa_{c}(G_{u},\B)$ is the restriction of some $\phi\in\sacgb$ by
  the Tietz Extension Theorem for upper semicontinuous Fell bundles
  \cite{muhwil:dm08}*{Proposition~A.5}.  So we can view the space of
  $\Ind\pi_{u}$ as the completion of $\sa_{c}(G_{u},\B)\odot
  \H_{\pi_{u}}$ with respect to~\eqref{eq:8}.
\end{example}

\section{The equivalence theorem}\label{sec:equivalence theorem}

Fix for this section second-countable locally compact Hausdorff
groupoids $G$ and $H$ with Haar systems $\lg$ and $\lh$, a $G\sme H$
equivalence $Z$, Fell bundles $\pg : \B \to G$ and $\ph : \CC \to H$
and a $\B\sme\CC$ equivalence $q : \E \to Z$.  Let $\ll$ denote the
Haar system on $L$ obtained from \cite{simwil:jot11}*{Remark~11}, and
let $\linkq : \linke \to L$ be the linking bundle of
Section~\ref{sec:linking bundle}.

\begin{thm}
  \label{thm-main-reduced}
  Suppose that $F \in \sa_{c}(L,\linke)$ satisfies $f(\zeta) = 0$ for
  all $\zeta \in L \setminus G$. Let $f := F|_G \in \sacgb$.  Then
  $\|F\|_{\cs(L,\linke)}=\|f\|_{\cs(G;\B)}$ and
  $\|F\|_{\cs_{r}(L;\linke)}=\|f\|_{\cs_{r}(G;\B)}$. Moreover, $\Pg
  \cs(L,\linke) \Ph$ is a $\cs(G;\B) \sme \cs(H;\CC)$-{\ib}, and $\Pg
  \cs_{r}(L,\linke) \Ph$ is a $\cs_{r}(G;\B) \sme
  \cs_{r}(H,\CC)$-{\ib} which is the quotient module of $\Pg
  \cs(L,\linke) \Ph$ by the kernel $I_r$ of the canonical homomorphism
  of $\cs(H,\CC)$ onto $\cs_{r}(H,\CC)$.
\end{thm}

\begin{remark}
  \label{rmk-scalar-case}
Recall the set-up of Example~\ref{ex-scalar}. It is not difficult to
see that if $G$ and $H$ are groupoids and $Z$ is a $G\sme H$
equivalence, then the trivial bundle $Z \times \C$ is a $(G \times
\C)\sme(H \times \C)$ equivalence. Hence we recover Theorem~13,
Proposition~15 and Theorem~17 of \cite{simwil:jot11} from
Theorem~\ref{thm-main-reduced}.
\end{remark}

To prove Theorem~\ref{thm-main-reduced}, we first establish some
preliminary results. Our key technical result is a norm-estimate for
the representations of $\sacgb \subseteq \sa_{c}(L,\linke)$ coming
from elements of $H^0$.

Let $\sset{\rho_{Z}^{v}}_{v\in\ho}$ be the Radon measures on $Z$
introduced in \cite{simwil:jot11}*{Theorem~13}.  For each $v\in\ho$,
fix $\zeta \in Z$ with $s(\zeta) = v$, and define a $D(v)$-valued form
on $\Y_{0}=\sa_{c}(Z,\E)$ by
\begin{equation}
  \label{eq:13}
  \ripd<\phi,\psi>(v)
  = \int_{G}\rip\CC<\phi(x^{-1}\cdot \zeta), \psi(x^{-1}\cdot
  \zeta)>\,d\lambda^{r(\zeta)} (x).
\end{equation}
Left-invariance of $\rho$ implies that this formula does not depend on
the choice of $\zeta\in Z$ such that $s(\zeta)=v$. The map
$(\phi,\psi) \mapsto \ripd<\phi,\psi>(v)$ is the restriction to $\lo$
of the product $\phi^*\psi$ computed in $\sa_c(L, \linke)$.

The following lemma constructs what is essentially an ``integrated
form'' of the modules used in \cite{moutu:xx11}*{Proposition~4.3}.

\begin{lemma}
  \label{lem-preD}
  With respect to the pre-inner product~\eqref{eq:13}, $\Y_{0}$ is a
  pre-Hilbert $D$-module whose completion, $\Y$ is a full right
  Hilbert $D$-module.
\end{lemma}
\begin{proof}
  That~\eqref{eq:13} takes values in $D$ follows from the observation
  above that $\ripd<\phi,\psi> = (\phi^*\psi)|_{\lo}$. Since each
  $E_z$ is an imprimitivity bimodule, the range of
  $\ripd<\cdot,\cdot>(v)$ is all of $D_v$. Since $D$ is a
  $C_0(\ho)$-algebra, to see that $X := \clsp\set{\ripd<\phi,\psi> :
    \phi,\psi \in \Y_{0}} = D$, it therefore suffices to show that $X$
  is a $C_0(\ho)$-module (see, for example
  \cite{wil:crossed}*{Proposition~C.24}), for which one uses the right
  action of $C_0(\ho)$ on $\Y_{0}$ to check that
  $(\ripd<\phi,\psi>\cdot f)(v) := f(v) \ripd<\phi,\psi>(v)$ is
  bilinear from $X \times C_0(\ho)$ to $X$. An argument like that of
  page~\pageref{page:multiplier} shows that for $f \in \sa_{c}(\lo,
  \linkq^{-1}(\lo))$, $M_f(\psi)(g) := f(r(g))\psi(g)$ determines a
  multiplier of $\cs(L; \linke)$, so $\Y_{0}$ is a pre-Hilbert
  $D$-module which is full since $X = D$.
\end{proof}

\begin{remark}
  \label{rem-norms}
  Since $D$ is a $C_{0}(\ho)$-algebra, to each $v\in\ho$ there
  corresponds a quotient module
  \begin{equation*}
    \Y(v):=Y/Y\cdot I_{v},
  \end{equation*}
  where $I_{v}=\set{d\in D:d(v)=0}$.  As in
  \cite{rw:morita}*{Proposition~3.25}, $\Y(v)$ is a right Hilbert
  $D(v)$-module: if we denote by $x(v)$ the image of $x$ in $\Y(v)$,
  then we have $\brip D(v)<x(v),y(v)>=\ripd <x,y>(v)$. Since
  $\|y\|^{2}=\|\ripd <y,y>\|$, we obtain
  $\|y\|=\sup_{v\in\ho}\|y(v)\|$. Indeed, the $\Y(v)$ are isomorphic
  to the modules used in \cite{moutu:xx11}*{Proposition~4.3}.

  Fix $T \in \L(\Y)$. Since $T$ is $D$-linear, for each $v \in \ho$
  there is an adjointable operator $T_{v}$ on $\Y(v)$ satisfying
  $T_{v}(x(v))=(Tx)(v)$ for all $x \in \Y$. Since $\brip
  D(v)<T_{v}(x(v)),y(v)> = \ripd <Tx,y>(v)$, we have
  $\|T\|=\sup_{v\in\ho}\|T_{v}\|$.
\end{remark}

\begin{prop}[\cite{moutu:xx11}*{Proposition~4.3}]
  \label{prop-key}
  There is a homomorphism $M$ from $\cs(G;\B)$ to $\L(\Y)$ such that if
  $f\in\sacgb$ and $\phi\in\Y_{0}$, then
  \begin{equation}
    \label{eq:14}
    M(f)\phi(\zeta)=\int_{G}f(x)\phi(x^{-1}\cdot
    \zeta)\,d\lambda^{r(\zeta)}(x) .
  \end{equation}
  We have $I_{\cs_{r}(G;\B)} \subset \ker M$, and $M$ factors through
  $\cs_{r}(G;\B)$.  In particular, $\|M(f)\|\le\|f\|_{r}$.
\end{prop}
\begin{proof}
  Direct computation shows that
  \begin{equation*}
    \ripd<M(f)\phi,\psi>=\ripd<\phi,M(f^{*})\psi>\quad\text{for
      $f\in\sacgb$ and $\phi,\psi\in\sa_{c}(Z,\E)$}.
  \end{equation*}
  A calculation using Remark~\ref{rem-norms}, the Cauchy-Schwartz
  inequality for Hilbert modules (\cite{rw:morita}*{Lemma~2.5}) and
  the characterization of inductive-limit topology
  continuous maps
  out of $C_{c}(G)$ in terms of eventually compactly supported uniform
  convergence
  shows that $f \mapsto \ripd<M(f) \phi,\psi>$ is continuous in the
  inductive-limit topology. The existence of approximate identities as
  in \cite{muhwil:dm08}*{Proposition~6.10} then implies that
  $\operatorname{span}\set{M(f)\phi:\text{$f\in\sacgb$ and
      $\phi\in\Y_{0}$}}$ is dense in $\Y$ in the inductive limit
  topology, so \cite{kmqw:nyjm10}*{Proposition~1.7} implies that $M$
  is bounded and extends to $\cs(G;\B)$.

  Since $\|M(f)\|=\sup_{v\in\ho}\|M_{v}(f)\|$, it now suffices to show
  that
  \begin{equation}
    \label{eq:15}
    \|M_{v}(f)\|\le\|f\|_{r}\quad\text{for all $v\in\ho$.}
  \end{equation}

  Fix $v\in\ho$ and choose $\zeta\in Z$ such that $s(\zeta)=v$.  Let
  $u=r(\zeta)$.  For any $e\in q^{-1}(\zeta)$ we have $\blip
  \B<\phi(x\cdot \zeta),e> \in B(\g[x\cdot \zeta,\zeta])=B(x)$.  Thus
  we can define
  \begin{equation*}
    U_{e}:\sa_{c}(Z,\E)\to \sa_{c}(G_{u},\B)\quad
    \text{by}\quad
    U(\phi)(x)=\blip \B<\phi(x\cdot \zeta),e>.
  \end{equation*}
  Just as in Remark~\ref{rem-norms}, we can form the quotient module
  $\X(u)$, and the map $V:\cs(G;\B)\to\L(\X)$ from
  Lemma~\ref{lem-x-red} gives operators $V_{u}(f)\in\L(\X(u))$ such
  that $\|V_{u}(f)\|\le\|f\|_{r}$.  The inner product
  $\Rip<\phi,\psi>(u)$ depends only on the $\phi\restr{G_{u}}$ and
  $\psi\restr{G_{u}}$, and every element of $\sa_{c}(G_{u},\B)$
  extends to an element of $\sacgb$ by the Tietz Extension Theorem for
  upper semicontinuous Banach bundles
  \cite{muhwil:dm08}*{Proposition~A.5}. So we can view $U_{e}$ as a
  map from $\sa_{c}(Z,\E)$ to $\X(u)$.

  Using that $\blip\B<\phi(x^{-1}\zeta),e>^{*}
  \blip\B<\phi(x^{-1}\cdot\zeta),e> = \blip\B<e\cdot
  {\rip\CC<\phi(x^{-1}\cdot \zeta),\phi(x^{-1}\cdot \zeta)>},e>$, one
  computes to see that
  \begin{equation}\label{eq:16}
    \Rip<U_{e}(\phi),U_{e}(\phi)>(u) = \brip\B<e\cdot \ripd<\phi,\phi>(v),e>.
  \end{equation}
  So if $\|e\| \le 1$, the Cauchy-Schwartz inequality for Hilbert
  modules (\cite{rw:morita}*{Lemma~2.5}) implies that
  $\|U_{e}(\phi)\|_{\X(u)}\le \|\phi\|_{\Y(v)}$.

  For $x\in G_{u}$, that the pairing $\blip\B<\cdot,\cdot>$ is
  $A$-linear in the first variable gives
  \[
  U_{e}(M(f)\phi)(x) = \int_{G}\;\blip\B<f(y)\phi(y^{-1}x\cdot
  \zeta),e>\,d\lambda^{r(x)}(y) = V_{u}(f)U(\phi)(x).
  \]

  Fix $f\in\sacgb$ and $\epsilon>0$. Fix $\phi\in\sa_{c}(Z,\E)$ such
  that $\|\phi\|_{\Y(v)}=1$ and such that
  $\|M(f)\phi\|_{\Y(v)}>\|M_{v}(f)\|-\epsilon$. By~\eqref{eq:16},
  there exists $e\in q^{-1}(\zeta)$ with $\|e\|=1$ such that
  \begin{equation*}
    \|U_{e}(M(f)\phi)\|_{\X(u)}>\|M(f)\phi\|_{\Y(v)}-\epsilon.
  \end{equation*}
  Hence
  \[
  \|M_{v}(f)\|-2\epsilon < \|M(f)\phi\|_{\Y(v)}-\epsilon <
  \|U_{e}(M(f)\phi)\|_{\X(u)} = \|V_{u}(f)U_{e}(\phi)\|_{\X(u)}
  \le\|f\|_{r}.
  \]
  Letting $\varepsilon \to 0$ gives~\eqref{eq:15}.
\end{proof}

\begin{proof}[Proof of Theorem~\ref{thm-main-reduced}]
  Since every representation of $\cs(L;\linke)$ restricts to a
  representation of $\cs(G;\B)$, we have $\|F\|_{\cs(L;\linke)} \le
  \|f\|_{\cs(G;\B)}$, so we just have to establish the reverse
  inequality. The argument for this is nearly identical to that of
  \cite{simwil:jot11}*{Proposition~15}. The key differences are that:
  \cite{muhwil:dm08}*{Theorem~6.4} is used in place of
  \cite{muhwil:nyjm08}*{Theorem~5.5}; and
  \cite{muhwil:dm08}*{Proposition~6.10} is used to obtain an
  approximate identity for both $\sacgb$ and $\sa_{c}(Z,\E)$ which can
  be used in place of the approximate identity in $\bundlefont{C}(L)$
  to establish the analogue of \cite{simwil:jot11}*{Equation~(10)} and
  to complete the norm approximation at the end of the proof.

  We now turn to the proof that the reduced norms agree. Fix faithful
  representations $\pi_u$ of the $A(u)$ and $\tau_v$ of the
  $D(v)$. Then
  \begin{equation*}
    \|F\|_{\cs_{r}(L;\linke)}
    = \max\Bigl\{\sup_{u\in\go}\|(\Ind^{L}\pi_{u})(F)\|, \sup_{v\in\ho}\|(\Ind\tau_{v})(F)\| \Bigr\}.
  \end{equation*}

  Fix $u\in\go$. Let $\H_{1}$ be the space of $\Ind^{G}\pi_{u}$; that
  is, the completion of $\sa_{c}(G_{u},\B)\odot\H_{\pi_{u}}$ as in
  Example~\ref{ex-pointmass}. The representation $\Ind^{L}\pi_{u}$
  acts on the completion of $\sa_{c}(L_{u},\linke)\odot \H_{\pi_{u}}$
  which decomposes as $\H_{1}\oplus\H_{2}$ where $\H_{1} =
    \overline{\sa_{c}(G_{u},\linke)\odot \H_{\pi_{u}}}$ and $\H_{2} =
    \overline{\sa_{c}(Z_{u}, \linke) \odot \H_{\pi_u}}$. Moreover, the
  restriction of $\Ind^L\pi_u$ to $H_2$ is the zero
  representation. Hence
  \begin{equation*}
    \|F\|_{\cs_{r}(L;\linke)} \ge \sup_{u \in \go}
    \|\Ind^{L}\pi_{u}(f)\| = \|f\|_{\cs_{r}(G;\B)},
  \end{equation*}
  and it suffices now to establish that $\|\Ind^L\tau_{v}(F)\| \le
  \sup_u\|\Ind^{L}\pi_{u}(F)\|$ for all $F \in \sa_{c}(L,\linke)$ and
  $v \in \ho$.

  Fix $v\in\ho$. Then $\Ind^{L}\tau_{v}$ acts on the completion of
  $\sa_{c}(L_{v},\linke)\odot \H_{\tau_{v}}$ which again decomposes as
  a direct sum $\H_{3}\oplus \H_{4}$ (here $\H_3 =
    \overline{\sa_c(\zop_{v}, \linke) \odot \H_{\tau_{v}}}$ and $\H_4 =
    \overline{\sa_c(H_v, \linke) \odot \H_{\tau_{v}}}$). The
  restriction to $\H_{4}$ is zero, and $\H_{3}$ is the completion of
  $\sa_{c}(Z,\E)$ under
  \begin{equation*}
    \ip(\phi\tensor h|\psi\tensor k)=\int_{Z}
    \bip(\tau_{v}\bigl(\brip\CC<\psi(\zeta), \phi(\zeta)>\bigr)h|k) \,
    d\rho^{v}(\zeta).
  \end{equation*}
  An inner-product computation shows that if $\Y$ is the Hilbert
  $D$-module of Lemma~\ref{lem-preD}, then $\H_{3}$ is isomorphic to
  the completion of $\Y\odot \H_{\tau_{v}}$ under
  \begin{equation*}
    \ip(\phi\tensor h|\psi\tensor k)
    =\bip(\tau_{v}\bigl(\ripd<\psi,\phi>(v)\bigr) h|k),
  \end{equation*}
  and then the restriction of $(\Ind^{L}\tau_{v})(F)$ to $\H_{3}$ is
  $\yind\tau_{v}$. Hence $\|\Ind^{L}\tau_{v}(F)\| =
  \|\yind\tau_{v}(f)\|$. Since $\yind\tau_{v}[x\tensor h] =
  [M(f)x\tensor h]$ for all $x$, we have $\ker M\subset \ker
  \yind\tau_{v}$. Hence Proposition~\ref{prop-key}, implies that
  $\|\yind\tau_{v}(f)\|\le\|f\|_{\cs_{r}(G;\B)}$ as required.

  The final statement follows from \cite{rw:morita}*{Theorem~3.22}.
\end{proof}

\section{The reduced symmetric imprimitivity theorem}

Suppose that $K$ and $H$ are locally compact groups acting freely and
properly on the left and right, respectively, of a locally compact
space $P$.  Suppose also that we have commuting actions $\alpha$ and
$\beta$ of $K$ and $H$, respectively, on a \cs-algebra $D$.  Then we
can form the induced algebras $\Ind_{H}^{P}(D,\beta)$ and
$\Ind_{K}^{P}(D,\alpha)$ and get dynamical systems
\begin{equation*}
  \check\sigma:K\to\Aut\bigl(\Ind_{H}^{P}(D,\beta)\bigr)
  \quad\text{and} \quad
  \check\tau:H\to\Aut\bigl(\Ind_{K}^{P}(D,\alpha)\bigr)
\end{equation*}
for the diagonal actions as in \cite{wil:crossed}*{Lemma~3.54}.  Then
Raeburn's Symmetric Imprimitivity Theorem says that the crossed
products
\begin{equation*}
  \Ind_{H}^{P}(D,\beta)\rtimes_{\check\sigma}K\quad\text{and}\quad
  \Ind_{K}^{P}(D,\alpha)\rtimes_{\check\tau} H
\end{equation*}
are Morita equivalent.  In \cite{quispi:hjm92}, Quigg and Spielberg
proved that Raeburn's Morita equivalence passed to the reduced crossed
products.  (Kasparov had a different proof in
\cite{kas:im88}*{Theorem~3.15} and an Huef and Raeburn gave a
different proof of the Quigg and Spielberg result in
\cite{huerae:mpcps00}*{Corollary~3}.)

We consider the corresponding statements for groupoid dynamical
systems. Let $(\A,G,\alpha)$ be a groupoid dynamical system as in
\cite{muhwil:nyjm08}*{\S4}. Recall that the associated crossed product
$\A\rtimes_{\alpha}G$ is a completion of $\sa_{c}(G,r^{*}\A)$.  If
$\pi$ is a representation of $A:=\sa_{0}(\go,\A)$, then the associated
regular representation of $\A\rtimes_{\alpha}G$ is the representation
$L^{\pi}:=\Ind^{\A}\pi$ acting on $L^{2}(G*\H_{s},\nu^{-1})$ as in
\eqref{eq:9}. The reduced crossed product, $\A\rtimes_{\alpha,r}G$ is
the quotient of $\A\rtimes_{\alpha} G$ by the common kernel of the
$L^{\pi}$ with $\pi$ faithful. Let $\B := r^{*}\A$ with the semidirect
product Fell bundle structure so that $\cs(G,\B)$ is isomorphic to
$\A\rtimes_{\alpha}G$, then it follows from
Example~\ref{ex-crossed-prod} that $\cs_{r}(G,\B)$ is isomorphic to
the reduced crossed product $\A\rtimes_{\alpha,r}G$.

Now let $H$ be a locally compact group and let $X$ be a locally
compact Hausdorff $H$-space. Let $G$ be the transformation groupoid $G
= H \times X$. As in \cite{muhwil:nyjm08}*{Example~4.8}, suppose that
$A=\sa_{0}(X,\A)$ is a $C_{0}(X)$-algebra, and define $\lt : H \to
\Aut(C_{0}(X))$ by $\lt_{h}(\phi)(x) := \phi(h^{-1}\cdot x)$. Suppose
that $\beta:H\to \Aut A$ is a \cs-dynamical system such that
\begin{equation*}
  \beta_{h}(\phi\cdot a)=\lt_{h}(\phi)\cdot \beta_{h}(a)\quad\text{for
    $h\in H$, $\phi\in C_{0}(X)$  and $a\in A$.}
\end{equation*}
Then, following \cite{muhwil:nyjm08}*{Example~4.8}, we obtain a
groupoid dynamical system $(\A,G,\alpha)$ where
\begin{equation*}
  \alpha_{(h,x)}\bigl(a(h^{-1}\cdot x)\bigr) =\beta_{h}(a)(x).
\end{equation*}
Let $\Delta$ be the modular function on $H$. Then the map
$\Phi:C_{c}(H,A)\to \sa_{c}(G,r^{*}\A)$ given by
\begin{equation*}
  \Phi(f)(h,x)=\Delta(h)^{\half}f(h)(x)
\end{equation*}
extends to an isomorphism of $A\rtimes_{\beta}H$ with
$\A\rtimes_{\alpha}G$.

Fix a representation $\pi$ of $A$. By decomposing $\pi$ as a direct
integral over $X$ one checks that $(f \tensor_\pi \xi | g \tensor_\pi
\eta) = (\Phi(f) \tensor_\pi \xi | \Phi(g) \tensor_\pi \eta)$ for $f,g
\in C_{c}(G,A)$ and $\xi,\eta \in \H_{\pi}$.  We use this to show that
$U(f\tensor h)=\Phi(f)\tensor h$ determines a unitary from the space
of the regular representation $\Ind^{A}\pi$ of $A\rtimes_{\beta}H$ to
the space of the regular representation $\Ind^{\A}\pi$ of
$\A\rtimes_{\alpha}G$ which intertwines $\Ind^{A}\pi(f)$ and
$\Ind^{\A} \pi(\Phi(f))$ for all $f$. Therefore $\Phi$ factors through
an isomorphism $A\rtimes_{\beta,r}H \cong \A\rtimes_{\alpha,r}G$.

Now, back to the set-up of Raeburn's Symmetric Imprimitivity
Theorem. Since $\Ind_{H}^{P}(D,\beta)$ is a $C_{0}(P/H)$-algebra, it
is the section algebra of a bundle $\B$ over $P/H$. It is shown in
\cite{muhwil:nyjm08}*{Example~5.12} that there is a groupoid action
$\sigma$ of the transformation groupoid $K\times P/H$ on $\B$ such
that
\begin{equation*}
  \Ind_{H}^{P}(D,\beta)\rtimes_{\check\sigma}K\cong
  \B\rtimes_{\sigma}(K\times P/H).
\end{equation*}
Similarly,
\begin{equation*}
  \Ind_{K}^{P}(D,\alpha)\rtimes_{\check\tau}H\cong
  \A\rtimes_{\tau}(K\backslash P\rtimes H)
\end{equation*}
for an appropriate bundle $\A$ over $K\backslash P$ and action $\tau$.
Furthermore, the trivial bundle $\E:=P\times A$ is an equivalence
between $(\B,K\times P/H,\sigma) $ and $(\A,K\backslash P\times
H,\tau)$ in the sense of \cite{muhwil:nyjm08}*{Definition~5.1}.  (Thus
Raeburn's Symmetric Imprimitivity Theorem is a special case of
\cite{muhwil:nyjm08}*{Theorem~5.5}.)  Therefore the Quigg-Spielberg
result follows from following corollary of our main theorem.

\begin{cor}
  \label{cor--redcrossprod}
  Suppose that $q:\E\to Z$ is an equivalence between the groupoid
  dynamical systems $(\B,H,\beta)$ and $(\A,G,\alpha)$.  Then the
  Morita equivalence of \cite{muhwil:nyjm08}*{Theorem~5.5} factors
  through a Morita equivalence of the reduced crossed products
  $\B\rtimes_{\beta,r}H$ and $\A\rtimes_{\alpha,r}G$.
\end{cor}
\begin{proof}
  Recall that $r^*\A := \{(a,x) \in \A \times G : r(a) = r(x)\}$ is a
  Fell bundle over $G$ with bundle map $(a,x) \mapsto x$,
  multiplication $(a,x)(b,y) = (a\alpha_x(b), xy)$ and involution
  $(a,x)^* = (\alpha_{x^{-1}}(a), x^{-1})$ (see
  \cite{muhwil:dm08}*{Example~2.1}), and similarly for $r*\B$. Define
  maps $r^*\A * \E \to \E$ and $\E * r^*\B \to \E$ by
  \begin{equation*}
    (a,g)\cdot \ea := a \cdot\alpha_{g}(\ea)\quad\text{ and }\quad \ea
    \cdot (b,h) := \beta_h(\ea)\cdot b.
  \end{equation*}
  Define pairings $\lip{r^*\A}<\cdot,\cdot> : \E *_s \E \to \A$ and
  $\rip{r^*\B}<\cdot,\cdot> : \E *_r \E \to \B$ by
  \begin{align*}
    \lip{r^*\A}<\ea,\eb> &=
    \big(\rip{A_{r(\ea)}}<\ea,\alpha_{\g[q(\ea),q(\eb)]}(\eb)>,
    \g[q(\ea),q(\eb)]\big)\quad\text{ and } \\
    \rip{r^*\B}<\ea,\eb> &=
    \big(\rip{A_{r(\ea)}}<\ea,\beta_{\h[q(\ea),q(\eb)]}(\eb)>,
    \h[q(\ea),q(\eb)]\big).
  \end{align*}
  It is routine though tedious to show that $\E$ is an $r^*\A\sme
  r^*\B$ equivalence.

  The Morita equivalence $X$ of \cite{muhwil:nyjm08}*{Theorem~5.5} and
  the Morita equivalence $\Pg C^*(L,\linke) \Ph$ of
  Theorem~\ref{thm-main-reduced} are both completions of
  $\sa_{c}(Z,\E)$. From the formulae for the actions of $r^*\A$ on
  $\E$, we see that the identity map on $\sa_{c}(Z,E)$ determines a
  left-module map from $X$ to $\Pg C^*(L,\linke) \Ph$, and similarly
  on the right. So it suffices to show that the norms on $X$ and on
  $\Pg C^*(L,\linke) \Ph$ coincide. For this, observe that the formula
  \cite{muhwil:nyjm08}*{Equation~(5.1)} for the $\A\times_\alpha
  G$-valued inner-product on $\sa_{c}(Z,\E)$ is precisely the
  convolution formula for multiplication of the corresponding elements
  of $\sa_{c}(L,\linke)$ with respect to the Haar system $\ll$
  described in \cite{simwil:jot11}.
\end{proof}

%%%%%%%%%%%%%%%%%%%%%%%%%%%%%%%%%%%%%%%%%%%%%%%%%%%%
%%%% End Matter
%%%%%%%%%%%%%%%%%%%%%%%%%%%%%%%%%%%%%%%%%%%%%%%%%%%%%
\def\noopsort#1{}\def\cprime{$'$} \def\sp{^}


\begin{bibdiv}
\begin{biblist}

\bib{bmz}{unpublished}{
      author={Buss, A.},
      author={Meyer, R.},
      author={Zhu, C.},
       title={A higher category approach to twisted actions on
  {$C^*$}-algebras},
     address={preprint},
        date={2009},
        note={(arXiv:math.OA.0908.0455v1)},
}

\bib{dauhof:mams68}{article}{
      author={Dauns, John},
      author={Hofmann, Karl~H.},
       title={Representation of rings by sections},
        date={1968},
     journal={Mem. Amer. Math. Soc.},
      volume={83},
       pages={1\ndash 180},
}

\bib{dg:banach}{book}{
      author={Dupr{\'e}, Maurice~J.},
      author={Gillette, Richard~M.},
       title={Banach bundles, {B}anach modules and automorphisms of
  ${C}^*$-algebras},
   publisher={Pitman (Advanced Publishing Program)},
     address={Boston, MA},
        date={1983},
      volume={92},
        ISBN={0-273-08626-X},
      review={\MR{85j:46127}},
}

\bib{fd:representations1}{book}{
      author={Fell, James M.~G.},
      author={Doran, Robert~S.},
       title={Representations of {$*$}-algebras, locally compact groups, and
  {B}anach {$*$}-algebraic bundles. {V}ol. 1},
      series={Pure and Applied Mathematics},
   publisher={Academic Press Inc.},
     address={Boston, MA},
        date={1988},
      volume={125},
        ISBN={0-12-252721-6},
        note={Basic representation theory of groups and algebras},
      review={\MR{90c:46001}},
}

\bib{fd:representations2}{book}{
      author={Fell, James M.~G.},
      author={Doran, Robert~S.},
       title={Representations of {$*$}-algebras, locally compact groups, and
  {B}anach {$*$}-algebraic bundles. {V}ol. 2},
      series={Pure and Applied Mathematics},
   publisher={Academic Press Inc.},
     address={Boston, MA},
        date={1988},
      volume={126},
        ISBN={0-12-252722-4},
        note={Banach $*$-algebraic bundles, induced representations, and the
  generalized Mackey analysis},
      review={\MR{90c:46002}},
}

\bib{hof:74}{unpublished}{
      author={Hofmann, Karl~Heinrich},
       title={Banach bundles},
        date={1974},
        note={Darmstadt Notes},
}

\bib{hof:lnm77}{incollection}{
      author={Hofmann, Karl~Heinrich},
       title={Bundles and sheaves are equivalent in the category of {B}anach
  spaces},
        date={1977},
   booktitle={{$K$}-theory and operator algebras ({P}roc. {C}onf., {U}niv.
  {G}eorgia, {A}thens, {G}a., 1975)},
      series={Lecture Notes in Math},
      volume={575},
   publisher={Springer},
     address={Berlin},
       pages={53\ndash 69},
      review={\MR{58 \#7117}},
}

\bib{hofkei:lnm79}{incollection}{
      author={Hofmann, Karl~Heinrich},
      author={Keimel, Klaus},
       title={Sheaf-theoretical concepts in analysis: bundles and sheaves of
  {B}anach spaces, {B}anach {$C(X)$}-modules},
        date={1979},
   booktitle={Applications of sheaves ({P}roc. {R}es. {S}ympos. {A}ppl. {S}heaf
  {T}heory to {L}ogic, {A}lgebra and {A}nal., {U}niv. {D}urham, {D}urham,
  1977)},
      series={Lecture Notes in Math.},
      volume={753},
   publisher={Springer},
     address={Berlin},
       pages={415\ndash 441},
      review={\MR{MR555553 (81f:46085)}},
}

\bib{huerae:mpcps00}{article}{
      author={Huef, Astrid~an},
      author={Raeburn, Iain},
       title={Regularity of induced representations and a theorem of {Q}uigg
  and {S}pielberg},
        date={2002},
        ISSN={0305-0041},
     journal={Math. Proc. Cambridge Philos. Soc.},
      volume={133},
      number={2},
       pages={249\ndash 259},
      review={\MR{2003g:46080}},
}

\bib{ionwil:pams08}{article}{
      author={Ionescu, Marius},
      author={Williams, Dana~P.},
       title={Irreducible representations of groupoid {$C\sp *$}-algebras},
        date={2009},
        ISSN={0002-9939},
     journal={Proc. Amer. Math. Soc.},
      volume={137},
      number={4},
       pages={1323\ndash 1332},
      review={\MR{MR2465655}},
}

\bib{kmqw:nyjm10}{article}{
      author={Kaliszewski, Steven},
      author={Muhly, Paul~S.},
      author={Quigg, John},
      author={Williams, Dana~P.},
       title={Coactions and {F}ell bundles},
        date={2010},
     journal={New York J. Math.},
      volume={16},
       pages={315\ndash 359},
}

\bib{kas:im88}{article}{
      author={Kasparov, Gennadi~G.},
       title={Equivariant {$KK$}-theory and the {N}ovikov conjecture},
        date={1988},
        ISSN={0020-9910},
     journal={Invent. Math.},
      volume={91},
      number={1},
       pages={147\ndash 201},
      review={\MR{88j:58123}},
}

\bib{kum:pams98}{article}{
      author={Kumjian, Alex},
       title={Fell bundles over groupoids},
        date={1998},
        ISSN={0002-9939},
     journal={Proc. Amer. Math. Soc.},
      volume={126},
      number={4},
       pages={1115\ndash 1125},
      review={\MR{MR1443836 (98i:46055)}},
}

\bib{moutu:xx11}{unpublished}{
      author={Moutuou, El-Ka\"ioum~M.},
      author={Tu, Jean-Louis},
       title={Equivalence of fell systems and their reduced {$C^*$}-algebras},
        date={2011},
        note={(arXiv:math.OA.1101.1235v1)},
}

\bib{mrw:jot87}{article}{
      author={Muhly, Paul~S.},
      author={Renault, Jean~N.},
      author={Williams, Dana~P.},
       title={Equivalence and isomorphism for groupoid {$C^*$}-algebras},
        date={1987},
        ISSN={0379-4024},
     journal={J. Operator Theory},
      volume={17},
      number={1},
       pages={3\ndash 22},
      review={\MR{88h:46123}},
}

\bib{muhwil:dm08}{article}{
      author={Muhly, Paul~S.},
      author={Williams, Dana~P.},
       title={Equivalence and disintegration theorems for {F}ell bundles and
  their {$C\sp *$}-algebras},
        date={2008},
        ISSN={0012-3862},
     journal={Dissertationes Math. (Rozprawy Mat.)},
      volume={456},
       pages={1\ndash 57},
      review={\MR{MR2446021}},
}

\bib{muhwil:nyjm08}{book}{
      author={Muhly, Paul~S.},
      author={Williams, Dana~P.},
       title={Renault's equivalence theorem for groupoid crossed products},
      series={NYJM Monographs},
   publisher={State University of New York University at Albany},
     address={Albany, NY},
        date={2008},
      volume={3},
        note={Available at http://nyjm.albany.edu:8000/m/2008/3.htm},
}

\bib{pim:fic87}{incollection}{
      author={Pimsner, Michael~V.},
       title={A class of {$C^*$}-algebras generalizing both {C}untz-{K}rieger
  algebras and crossed products by {${\bf Z}$}},
        date={1997},
   booktitle={Free probability theory ({W}aterloo, {ON}, 1995)},
      series={Fields Inst. Commun.},
      volume={12},
   publisher={Amer. Math. Soc.},
     address={Providence, RI},
       pages={189\ndash 212},
      review={\MR{1426840 (97k:46069)}},
}

\bib{quispi:hjm92}{article}{
      author={Quigg, John},
      author={Spielberg, Jack},
       title={Regularity and hyporegularity in {\cs}-dynamical systems},
        date={1992},
     journal={Houston J. Math.},
      volume={18},
       pages={139\ndash 152},
}

\bib{rae:ma88}{article}{
      author={Raeburn, Iain},
       title={Induced {$C\sp *$}-algebras and a symmetric imprimitivity
  theorem},
        date={1988},
        ISSN={0025-5831},
     journal={Math. Ann.},
      volume={280},
      number={3},
       pages={369\ndash 387},
      review={\MR{90k:46144}},
}

\bib{rw:morita}{book}{
      author={Raeburn, Iain},
      author={Williams, Dana~P.},
       title={Morita equivalence and continuous-trace {$C^*$}-algebras},
      series={Mathematical Surveys and Monographs},
   publisher={American Mathematical Society},
     address={Providence, RI},
        date={1998},
      volume={60},
        ISBN={0-8218-0860-5},
      review={\MR{2000c:46108}},
}

\bib{simwil:jot11}{article}{
      author={Sims, Aidan},
      author={Williams, Dana~P.},
       title={Renault's equivalence theorem for reduced groupoid
  {$C^*$}-algebras},
        date={2012},
     journal={J. Operator Theory},
       pages={in press},
        note={(arXiv:math.OA.1002.3093)},
}

\bib{wil:crossed}{book}{
      author={Williams, Dana~P.},
       title={Crossed products of {$C{\sp \ast}$}-algebras},
      series={Mathematical Surveys and Monographs},
   publisher={American Mathematical Society},
     address={Providence, RI},
        date={2007},
      volume={134},
        ISBN={978-0-8218-4242-3; 0-8218-4242-0},
      review={\MR{MR2288954 (2007m:46003)}},
}

\end{biblist}
\end{bibdiv}
\end{document}